\documentclass[11pt]{article}
\usepackage[margin=1in]{geometry}
\usepackage{amsmath,amssymb,amsthm,mathtools}
\usepackage{tikz}
\usepackage{hyperref}

\newtheorem{theorem}{Theorem}
\newtheorem{corollary}[theorem]{Corollary}
\newtheorem{proposition}[theorem]{Proposition}
\newtheorem{lemma}[theorem]{Lemma}
\theoremstyle{definition}
\newtheorem{example}[theorem]{Example}
\newtheorem{remark}[theorem]{Remark}

\newcommand{\R}{\mathbb R}
\newcommand{\aff}{\operatorname{aff}}
\newcommand{\conv}{\operatorname{conv}}

\title{Corner Volumes and the Steiner Ellipses of a Simplex}
\author{Anatoly Eydelzon\\
Department of Mathematical Sciences\\
The University of Texas at Dallas\\
Richardson, Texas 75080, USA\\
\texttt{anatoly@utdallas.edu}}
\date{}

\begin{document}
\maketitle

\begin{abstract}
Through a point $M$ draw the hyperplanes parallel to the facets of an
$n$-simplex $S$ of volume $V$. Together with the facets they determine
$n+1$ small simplices with a corner at $M$, of volumes
$V_1,\ldots,V_{n+1}$. We show that
$\sum_i V_i^{2/n}=\tfrac1n V^{2/n}$ describes the boundary of the John
ellipsoid of $S$ (the inscribed ellipsoid of maximal volume), and that
$\sum_i V_i^{2/n}=V^{2/n}$ describes the boundary of the L\"owner ellipsoid
(the circumscribed ellipsoid of minimal volume). In the plane these read
$T_1+T_2+T_3=\tfrac12T$ for the Steiner inellipse and $T_1+T_2+T_3=T$ for
the Steiner circumellipse. Other levels of the same quantity give the
ellipsoids tangent to the $k$-faces at their centroids and the Legendre
ellipsoid.
\end{abstract}



\section{Introduction}

Every triangle carries two distinguished ellipses with center at the
centroid. The \emph{Steiner inellipse} touches the sides at their midpoints
and is the inscribed ellipse of maximal area; the \emph{Steiner
circumellipse} passes through the vertices and is the circumscribed ellipse
of minimal area \cite{Chakerian1979}. In normalized barycentric coordinates
$(x,y,z)$, $x+y+z=1$, their classical equations are
\begin{equation}\label{eq:classical}
x^2+y^2+z^2-2(xy+yz+zx)=0
\qquad\text{and}\qquad
xy+yz+zx=0,
\end{equation}
respectively; see, e.g., \cite[\S\S10.2.1, 11.3.1]{Yiu}. Because
$(x+y+z)^2=1$, these equations can be rewritten as
\begin{equation}\label{eq:classical-squares}
x^2+y^2+z^2=\tfrac12
\qquad\text{and}\qquad
x^2+y^2+z^2=1 .
\end{equation}

The squares in \eqref{eq:classical-squares} have a simple geometric
meaning. Through a point $M$ draw the three lines parallel to the sides of a
triangle of area $T$. They cut off three small triangles that have a corner
at $M$ and stand on the three sides (Figure~\ref{fig:triangle}). Each of
them is homothetic to the whole triangle, with ratio equal to a barycentric
coordinate of $M$, so its area is the \emph{square} of that coordinate
times $T$. Thus \eqref{eq:classical-squares} becomes the statement, proved
in \cite{Eydelzon2020} for the inellipse, that $M$ lies on the Steiner
inellipse if and only if the three small triangles have total area
\begin{equation}\label{eq:old}
T_1+T_2+T_3=\tfrac12 T,
\end{equation}
and, as we shall see, $M$ lies on the Steiner circumellipse if and only if
$T_1+T_2+T_3=T$.

In dimension $n$ the Steiner inellipse and circumellipse have natural
successors: the John ellipsoid $E_J(S)$ of a simplex $S$, the unique
ellipsoid of maximal volume contained in $S$, and the L\"owner ellipsoid
$E_L(S)$, the unique ellipsoid of minimal volume containing $S$
\cite{John1948,Ball1992,Ball1997,Henk2012}. Several planar facts about the
Steiner inellipse carry over to them. For instance, $E_J(S)$ touches each
facet at its centroid, and $E_L(S)$ passes through the vertices
\cite{LinGeLeng2006,Henk2012}; a real, higher-dimensional analogue of the
Siebeck--Marden description of the foci is given in \cite{EydelzonFocal}.
The purpose of this note is to show that the corner-area identity
\eqref{eq:old} carries over as well, with one change: areas are replaced by
volumes raised to the power $2/n$.

\begin{figure}[t]
\centering
\begin{tikzpicture}[scale=0.85]
  \fill[gray!25] (1,1)--(5,1)--(1,5)--cycle;
  \fill[gray!45] (1,1)--(0,1)--(0,2)--cycle;
  \fill[gray!45] (1,1)--(1,0)--(2,0)--cycle;
  \draw[thin] (1,0)--(1,5) (0,1)--(5,1) (0,2)--(2,0);
  \draw[thick,rotate around={-45:(2,2)}] (2,2) ellipse [x radius=4.899, y radius=2.828];
  \draw[dashed,rotate around={-45:(2,2)}] (2,2) ellipse [x radius=3.464, y radius=2];
  \draw[thick,rotate around={-45:(2,2)}] (2,2) ellipse [x radius=2.449, y radius=1.414];
  \draw[very thick] (0,0)--(6,0)--(0,6)--cycle;
  \fill (0,0) circle (1.6pt) node[below left] {$P_1$};
  \fill (6,0) circle (1.6pt) node[below right] {$P_2$};
  \fill (0,6) circle (1.6pt) node[above left] {$P_3$};
  \fill (2,2) circle (1.6pt) node[above right] {$G$};
  \fill (1,1) circle (1.6pt) node[above right=-1pt] {$M$};
  \node at (2.3,2.9) {$T_1$};
  \node at (0.32,1.35) {\small $T_2$};
  \node at (1.35,0.3) {\small $T_3$};
\end{tikzpicture}
\caption{The right triangle of Example~\ref{ex:triangle}, with the three
corner triangles of the point $M=(\tfrac16,\tfrac16)$ on the Steiner
inellipse. Here $T_1+T_2+T_3=\tfrac12T$. The outer ellipse is the Steiner
circumellipse ($T_1+T_2+T_3=T$); the dashed ellipse is the Legendre ellipse
($T_1+T_2+T_3=\tfrac23T$).}
\label{fig:triangle}
\end{figure}

\section{Corner simplices}

Let $S=\conv(P_1,\ldots,P_{n+1})\subset\R^n$ be a nondegenerate simplex of
volume $V$, let $F_i$ be the facet opposite $P_i$, and let
$\lambda_1(M),\ldots,\lambda_{n+1}(M)$ be the barycentric coordinates of a
point $M\in\R^n$, so that
$M=\sum_i\lambda_i(M)P_i$ and $\sum_i\lambda_i(M)=1$. Each $\lambda_i$ is an
affine function on $\R^n$, it vanishes on $\aff F_i$, and $\lambda_i(P_j)$
equals $1$ for $j=i$ and $0$ otherwise. Coordinates may be negative when
$M\notin S$.

For $M\in\R^n$ and $1\le i\le n+1$ put
\[
S_i(M)=M+\lambda_i(M)\,(S-P_i),
\qquad
V_i(M)=\operatorname{vol}S_i(M).
\]

\begin{lemma}\label{lem:corner}
If $\lambda_i(M)\ne0$, then $S_i(M)$ is the simplex bounded by $\aff F_i$
and the $n$ hyperplanes through $M$ parallel to the facets $F_j$, $j\ne i$;
its vertex corresponding to $P_i$ is $M$. In every case
\begin{equation}\label{eq:corner}
\frac{V_i(M)}{V}=|\lambda_i(M)|^n,
\qquad\text{hence}\qquad
\left(\frac{V_i(M)}{V}\right)^{2/n}=\lambda_i(M)^2 .
\end{equation}
\end{lemma}

\begin{proof}
The map $h(X)=M+\lambda_i(M)(X-P_i)$ is a homothety (followed by a point
reflection if $\lambda_i(M)<0$). It sends $P_i$ to $M$ and each hyperplane
$\aff F_j$ through $P_i$ ($j\ne i$) to the parallel hyperplane through $M$.
For $j\ne i$, the linear part of $\lambda_i$ takes the value
$\lambda_i(P_j)-\lambda_i(P_i)=-1$ on $P_j-P_i$, so
$\lambda_i(h(P_j))=\lambda_i(M)-\lambda_i(M)=0$. Hence $h$ maps $\aff F_i$
onto itself. The volume formula follows since $h$ scales volumes by
$|\lambda_i(M)|^n$.
\end{proof}

We call $S_i(M)$ the \emph{corner simplices} of $M$: each has a corner at
$M$ and stands on a facet of $S$. For $n=2$ and $M$ inside the triangle they
are the three small triangles of Figure~\ref{fig:triangle}. (They should not
be confused with the simplex cut off at the vertex $P_i$ by the hyperplane
through $M$ parallel to $F_i$, whose ratio is $1-\lambda_i(M)$.) Lemma
\ref{lem:corner} explains the exponent $2/n$: the $n$th root of $V_i/V$ is a
linear ratio, and its square is the quadratic quantity whose level sets are
ellipsoids.

\section{The main result}

Define, for every $M\in\R^n$,
\[
\Phi_S(M)=\sum_{i=1}^{n+1}\left(\frac{V_i(M)}{V}\right)^{2/n}
=\sum_{i=1}^{n+1}\lambda_i(M)^2 ,
\]
the second expression by \eqref{eq:corner}.

\begin{theorem}\label{thm:main}
Let $S$ be a nondegenerate $n$-simplex of volume $V$, and let $M\in\R^n$.
Then
\begin{align*}
M\in\partial E_J(S)
&\quad\Longleftrightarrow\quad
\sum_{i=1}^{n+1}V_i(M)^{2/n}=\frac1n\,V^{2/n},\\
M\in\partial E_L(S)
&\quad\Longleftrightarrow\quad
\sum_{i=1}^{n+1}V_i(M)^{2/n}=V^{2/n}.
\end{align*}
Moreover, $M$ lies in the interior of $E_J(S)$ if and only if
$\sum_{i}V_i(M)^{2/n}<\frac1n\,V^{2/n}$, and $M$ lies in the interior of
$E_L(S)$ if and only if $\sum_{i}V_i(M)^{2/n}<V^{2/n}$.
\end{theorem}

\begin{proof}
Write $\lambda_i=\frac1{n+1}+\mu_i$ with $\sum_i\mu_i=0$. Then
\begin{equation}\label{eq:mu}
\Phi_S(M)=\frac1{n+1}+\sum_i\mu_i^2 .
\end{equation}
First let $S$ be regular with centroid $G$ and circumradius $R$. The vectors
$v_i=P_i-G$ satisfy $|v_i|^2=R^2$ and $v_i\cdot v_j=-R^2/n$ for $i\ne j$.
Since $M-G=\sum_i\mu_iv_i$ and $\sum_i\mu_i=0$,
\[
|M-G|^2=R^2\sum_i\mu_i^2-\frac{R^2}{n}\sum_{i\ne j}\mu_i\mu_j
=\frac{n+1}{n}\,R^2\sum_i\mu_i^2 .
\]
So the level sets of $\Phi_S$ are the spheres centered at $G$. By uniqueness
and symmetry, $E_J(S)$ is the insphere, which passes through the centroids
of the facets, and $E_L(S)$ is the circumsphere, which passes through the
vertices. A facet centroid has one coordinate $0$ and $n$ coordinates
$\frac1n$, so $\Phi_S=n\cdot\frac1{n^2}=\frac1n$ there; a vertex has
$\Phi_S=1$. This proves the theorem for regular simplices.

Every nondegenerate simplex is the image of a regular one under an affine
bijection $A$. Such a map preserves barycentric coordinates, hence
$\Phi_S$; it maps the John and L\"owner ellipsoids of the regular simplex to
those of $S$, because it multiplies all volumes by the same factor
$|\det A|$. The theorem follows for all simplices.
\end{proof}

For $n=2$ the exponent $2/n$ is $1$, and we recover \eqref{eq:old}
together with its circumscribed companion:
\begin{corollary}\label{cor:plane}
For a triangle of area $T$ and a point $M$ with corner triangles of areas
$T_1,T_2,T_3$,
\begin{align*}
M\ \text{lies on the Steiner inellipse}
&\iff T_1+T_2+T_3=\tfrac12 T,\\
M\ \text{lies on the Steiner circumellipse}
&\iff T_1+T_2+T_3=T .
\end{align*}
\end{corollary}

\begin{corollary}\label{cor:tetra}
For a tetrahedron of volume $V$ and a point $M$ with corner tetrahedra of
volumes $V_1,\ldots,V_4$,
\[
M\in\partial E_J(S)\iff \sum_{i=1}^4 V_i^{2/3}=\tfrac13V^{2/3},
\qquad
M\in\partial E_L(S)\iff \sum_{i=1}^4 V_i^{2/3}=V^{2/3}.
\]
\end{corollary}

\section{More ellipsoids from the same quantity}

By \eqref{eq:mu}, $\Phi_S$ is a positive definite quadratic function with
minimum $\frac1{n+1}$ at the centroid $G$, so every set
$\{\Phi_S=c\}$ with $c>\frac1{n+1}$ is an ellipsoid centered at $G$. The
John and L\"owner ellipsoids are the levels $\frac1n$ and $1$. Two further
families of classical ellipsoids also appear.

\begin{proposition}\label{prop:faces}
Let $0\le k\le n-1$. The ellipsoid $\{\Phi_S=\frac1{k+1}\}$ passes through the
centroids of all $k$-faces of $S$, and for $k\ge1$ it is tangent there to
the affine hull of each $k$-face.
\end{proposition}

\begin{proof}
On the affine hull of a $k$-face only $k+1$ coordinates are nonzero, and
they sum to $1$. Their sum of squares is at least $\frac1{k+1}$, with
equality only at the centroid of the face. So the affine hull meets the
solid ellipsoid $\{\Phi_S\le\frac1{k+1}\}$ in exactly one point, the
centroid, and is tangent to it there.
\end{proof}

For $k=n-1$ and $k=0$ this is the John and the L\"owner ellipsoid. For
$n=3$ and $k=1$ we get the ellipsoid tangent to all six edges of a
tetrahedron at their midpoints, the affine image of the midsphere of a
regular tetrahedron: $\sum_i V_i^{2/3}=\frac12V^{2/3}$.

The second family comes from second moments. Let $C$ be the covariance
matrix of the uniform distribution on $S$.

\begin{proposition}\label{prop:moment}
For every $M\in\R^n$,
\[
\Phi_S(M)=\frac1{n+1}+\frac1{(n+1)(n+2)}\,(M-G)^TC^{-1}(M-G).
\]
\end{proposition}

\begin{proof}
Let $B$ be the $n\times(n+1)$ matrix with columns $v_i=P_i-G$. Then
$M-G=B\mu$, where $\mu=(\mu_i)$ is orthogonal to $(1,\ldots,1)$, the kernel
of $B$. Hence $\mu=B^T(BB^T)^{-1}(M-G)$ and
$\sum_i\mu_i^2=(M-G)^T(BB^T)^{-1}(M-G)$. If $X$ is uniform on $S$ with
barycentric coordinates $\Lambda_i$, then $X-G=\sum_i\Lambda_iv_i$, and the
standard moments $\mathbb E\Lambda_i^2=\frac{2}{(n+1)(n+2)}$,
$\mathbb E\Lambda_i\Lambda_j=\frac1{(n+1)(n+2)}$ ($i\ne j$), together with
$\sum_iv_i=0$, give
\[
C=\frac{1}{(n+1)(n+2)}\Bigl(\sum_iv_iv_i^T+\Bigl(\sum_iv_i\Bigr)\Bigl(\sum_iv_i\Bigr)^T\Bigr)
=\frac{1}{(n+1)(n+2)}\,BB^T .
\]
Substituting into \eqref{eq:mu} proves the claim.
\end{proof}

The \emph{Legendre ellipsoid} of $S$ is the ellipsoid with the same
centroid and covariance matrix as $S$, namely
$\{(X-G)^TC^{-1}(X-G)\le n+2\}$ \cite{LYZ2000}. By
Proposition~\ref{prop:moment} it is the level $\Phi_S=\frac2{n+1}$. In the
plane the four levels line up neatly:
\[
\frac{T_1+T_2+T_3}{T}=
\frac13\ \text{(centroid)},\qquad
\frac12\ \text{(inellipse)},\qquad
\frac23\ \text{(Legendre ellipse)},\qquad
1\ \text{(circumellipse)}.
\]

\begin{remark}\label{rem:why}
It is no accident that so many ellipsoids appear. Suppose an ellipsoid
$E(S)$ is assigned to every simplex so that $E(AS)=A\,E(S)$ for affine
bijections $A$. For a regular simplex, $E(S)$ is invariant under all
symmetries of $S$; these fix only the centroid and act irreducibly on
$\R^n$, so $E(S)$ is a ball centered at $G$, that is, a level set of
$\Phi_S$. Affine invariance then shows that $E(S)=\{\Phi_S\le c\}$ for every
simplex, with a constant $c$ depending only on the construction. The
results above identify this constant for the John, L\"owner, face-centroid
and Legendre ellipsoids, and translate it into corner volumes.
\end{remark}

\begin{remark}\label{rem:root}
Taking $n$th roots instead of squares in \eqref{eq:corner} gives
$\sum_i\pm V_i(M)^{1/n}=V^{1/n}$, with the sign of $\lambda_i(M)$. For
$M\in S$ all signs are positive and
$\sum_iV_i(M)^{1/n}=V^{1/n}$; for $n=2$ this is the classical identity
$\sqrt{T_1}+\sqrt{T_2}+\sqrt{T_3}=\sqrt T$ \cite{Skanavi1980}. The first
powers of the barycentric coordinates always sum to $1$; it is their squares
that pick out the ellipsoids.
\end{remark}

\section{Two computed examples}

\begin{example}[The right triangle]\label{ex:triangle}
Let $P_1=(0,0)$, $P_2=(1,0)$, $P_3=(0,1)$, so $T=\frac12$ and
$(\lambda_1,\lambda_2,\lambda_3)=(1-x-y,\,x,\,y)$. Then
\[
\Phi_S(x,y)=(1-x-y)^2+x^2+y^2 ,
\]
and Corollary~\ref{cor:plane} gives the two Steiner ellipses directly:
\[
\Phi_S=\tfrac12:\quad x^2+xy+y^2-x-y+\tfrac14=0,
\qquad
\Phi_S=1:\quad x^2+xy+y^2-x-y=0 .
\]
For instance, $M=(\frac16,\frac16)$ has coordinates
$(\frac23,\frac16,\frac16)$, and its corner triangles have areas
$\frac29,\ \frac1{72},\ \frac1{72}$, with sum $\frac14=\frac12T$; so $M$ lies on
the inellipse (Figure~\ref{fig:triangle}).

Writing $X=G+u$ with $G=(\frac13,\frac13)$, one finds
$\Phi_S=\frac13+2(u_1^2+u_1u_2+u_2^2)$. The quadratic form has eigenvalues
$\frac12$ along $(1,-1)$ and $\frac32$ along $(1,1)$, so the inellipse has
squared semiaxes $\frac16$ and $\frac1{18}$, with the major axis parallel
to $(1,-1)$, and its foci lie at distance
$\sqrt{\frac16-\frac1{18}}=\frac13$ from $G$ on the line $x+y=\frac23$. This
agrees with the Siebeck--Marden theorem \cite{Kalman2008}: the zeros of
$p'(z)$ for $p(z)=z(z-1)(z-i)$ are
$\frac{1+i}{3}\pm\frac{1-i}{3\sqrt2}$. The circumellipse is the same
ellipse scaled by $2$ from $G$.
\end{example}

\begin{example}[The standard tetrahedron]\label{ex:tetra}
Let $P_1=0$ and $P_2,P_3,P_4$ be the unit vectors $e_1,e_2,e_3$ in $\R^3$,
so $V=\frac16$ and
$(\lambda_i)=(1-x-y-z,\,x,\,y,\,z)$. With
$Q(x,y,z)=x^2+y^2+z^2+xy+yz+zx$, we get
$\Phi_S=1-2(x+y+z)+2Q$, and the three ellipsoids of
Section~4 are
\begin{align*}
E_J(S)\ \ (\Phi_S=\tfrac13):&\quad Q=x+y+z-\tfrac13,\\
\text{edge-midpoint ellipsoid}\ \ (\Phi_S=\tfrac12):&\quad Q=x+y+z-\tfrac14,\\
E_L(S)\ \ (\Phi_S=1):&\quad Q=x+y+z .
\end{align*}
One checks directly that the face centroid
$(\frac13,\frac13,0)$, the edge midpoint $(\frac12,\frac12,0)$ and the vertex
$(1,0,0)$ satisfy the respective equations.

The point $M=(\frac16,\frac16,\frac16)$ has coordinates
$(\frac12,\frac16,\frac16,\frac16)$. Its corner tetrahedra have volumes
$\frac18V$ and three times $\frac1{216}V$, and
\[
\sum_{i=1}^4\Bigl(\frac{V_i}{V}\Bigr)^{2/3}
=\frac14+3\cdot\frac1{36}=\frac13 ,
\]
so $M$ lies on the John ellipsoid. In coordinates $X=G+u$ centered at
$G=(\frac14,\frac14,\frac14)$ we have $\Phi_S=\frac14+2Q(u)$. The form $Q$ has
eigenvalue $2$ along $(1,1,1)$ and $\frac12$ on the orthogonal plane, so the
John ellipsoid has squared semiaxes $\frac1{12},\frac1{12},\frac1{48}$, with
the short axis along $(1,1,1)$; the point $M$ above is an endpoint of this
short axis. The L\"owner ellipsoid is the same ellipsoid scaled by $3$ from
$G$.
\end{example}

\section*{Acknowledgment}
During the preparation of this paper the author used Claude (Anthropic) for
assistance with exposition, language, \LaTeX\ editing and formatting. The
author reviewed and edited all content and takes full responsibility for the
publication.

\end{document}